\documentclass[11pt]{article}
\usepackage{amsmath, amsthm, amssymb,amscd, amsfonts}

\usepackage{tkz-euclide, float}
\usetikzlibrary{shapes, calc, decorations.pathreplacing, calligraphy} 
\usepgflibrary{patterns, patterns.meta}
\usepackage{soul}

\usepackage{enumerate, enumitem, bbm, bm}

\usepackage{fullpage, verbatim, subcaption, array, tabularx, comment}
\usepackage{tikz, xcolor, graphicx}
\tikzstyle{uStyle}=[shape = circle, minimum size = 20pt, inner sep =2.5pt, outer sep = 0pt, draw, fill=white]
\tikzstyle{myStyle}=[shape = circle, draw, fill=black, scale=0.5]

\usepackage{url,hyperref}
\hypersetup{colorlinks=true,linkcolor=black,anchorcolor=blue,citecolor=black,filecolor=blue,urlcolor=blue,bookmarksnumbered=true,pdfview=FitB}

\usepackage{thmtools}
\usepackage{thm-restate}
\usepackage{soul}
\usepackage{cleveref}
\usepackage[all,arc]{xy}
\usepackage{mathrsfs}
\usepackage{mathtools}

\usepackage{todonotes}

\RequirePackage{marginnote,hyperref}
\newcommand{\aside}[1]{\marginnote{\scriptsize{#1}}[0cm]}
\newcommand{\aaside}[2]{\marginnote{\scriptsize{#1}}[#2]}
\newcommand\Emph[1]{\emph{#1}\aside{#1}}
\newcommand\EmphE[2]{\emph{#1}\aaside{#1}{#2}}

\newtheorem{lem}{Lemma}[section]

\newtheorem{conj}{Conjecture}
\newtheorem{thm}{Theorem}[section]

\newtheorem{claim}{Claim}

\newtheorem*{key-lem}{Key Lemma}

\def\diam{\textrm{diam}}

\def\max{\textrm{max}}

\title{Reconstruction and Edge Reconstruction of Triangle-free Graphs}
\author{Alexander Clifton\thanks{Discrete Mathematics Group, Institute for Basic Science, Daejeon, South Korea; \texttt{yoa@ibs.re.kr}} \and Xiaonan Liu\thanks{School of Mathematics, Georgia Institute of Technology, Atlanta, GA, USA; \texttt{xliu729@gatech.edu}} \and Reem Mahmoud\thanks{Department of Computer Science, Virginia Commonwealth University, Richmond, VA, USA; \texttt{mahmoudr@vcu.edu}} \and Abhinav Shantanam\thanks{Department of Mathematics, Simon Fraser University, Burnaby, BC, Canada; \texttt{ashantan@sfu.ca}}}

\date{\today}

\begin{document}

\maketitle

\begin{abstract}
The Reconstruction Conjecture due to Kelly and Ulam states that every graph with at least 3 vertices is uniquely determined by its multiset of subgraphs $\{G-v: v\in V(G)\}$. Let $diam(G)$ and $\kappa(G)$ denote the diameter and the connectivity of a graph $G$, respectively, and let $\mathcal{G}_2:=\{G: \diam(G)=2\}$ and $\mathcal{G}_3:=\{G:\diam(G)=\diam(\overline{G})=3\}$. It is known that the Reconstruction Conjecture is true if and only if it is true for every 2-connected graph in $\mathcal{G}_2\cup \mathcal{G}_3$. Balakumar and Monikandan showed that the Reconstruction Conjecture holds for every triangle-free graph $G$ in $\mathcal{G}_2\cup \mathcal{G}_3$ with $\kappa(G)=2$. Moreover, they asked whether the result still holds if $\kappa(G)\ge 3$. (If yes, the class of graphs critical for solving the Reconstruction Conjecture is restricted to 2-connected graphs in $\mathcal{G}_2\cup\mathcal{G}_3$ which contain triangles.) In this paper, we give a partial solution to their question by showing that the Reconstruction Conjecture holds for every triangle-free graph $G$ in $\mathcal{G}_3$ and every triangle-free graph $G$ in $\mathcal{G}_2$ with $\kappa(G)=3$. We also prove similar results about the Edge Reconstruction Conjecture.
\end{abstract}

\section{Introduction}
Throughout this paper, we use standard graph theory terminology and notation, as in \cite{We96}. Unless stated otherwise, assume $|V(G)|\ge3$ and $|E(G)|\ge4$ for every graph $G$. For vertices $u$ and $v$ in a graph $G$, we denote by \Emph{$d_G(u,v)$} the length of a shortest path from $u$ to $v$ in $G$. The \emph{diameter} of a graph $G$, denoted \Emph{$\diam(G)$}, is $\max_{u,v\in V(G)}d_G(u,v)$. We denote by $N_G(v)$ the neighborhood of a vertex $v$ in $G$. For a connected graph $G$, a set $S\subseteq V(G)$ is a \EmphE{cut set}{-1mm} if $G-S$ is disconnected; moreover, if $S=\{v\}$, then $v$ is a \EmphE{cut vertex}{-3mm}. The \Emph{connectivity} of $G$, denoted by \EmphE{$\kappa(G)$}{2.5mm}, is the size of its smallest cut set. For $k\geq2$, a graph $G$ is \EmphE{$k$-connected}{2mm} if its connectivity is at least $k$.

Graph Reconstruction is the study which explores whether a graph can be uniquely determined by its subgraphs. A \Emph{card} of a graph $G$ is a subgraph of $G$ obtained by deleting a single vertex; that is, $G-v$ for some $v\in V(G)$. The multiset $\mathcal{D}(G)$ of cards of $G$ is the \emph{deck}\aside{deck, $\mathcal{D}(G)$} of $G$, i.e., $\mathcal{D}(G):=\{G-v: v\in V(G)\}$. If $G$ is isomorphic to every graph $H$ with $\mathcal{D}(H)=\mathcal{D}(G)$, then $G$ is \EmphE{reconstructible}{-2mm}.
The most well-studied problem in the area of graph reconstruction is the \EmphE{Reconstruction Conjecture}{7mm} proposed by Ulam \cite{Ulam60} and Kelly \cite{Kel42, Kel57}. 
\begin{conj}[Reconstruction Conjecture]
For $n\ge 3$, every $n$-vertex graph is reconstructible, i.e., it is uniquely determined by its deck.
\end{conj}
This conjecture has attracted a lot of attention. It has been confirmed for certain graph classes such as disconnected graphs, trees, regular graphs, perfect graphs, etc. Although much work has gone into proving the conjecture, it remains widely open even for sparse classes of graphs such as bipartite graphs, planar graphs and graphs of bounded maximum degree. For a detailed survey of results on the Reconstruction Conjecture and graph reconstruction in general, we refer the reader to \cite{La13}.

Yang proved \cite{Yang88} that the Reconstruction Conjecture is true if and only if it is true for every 2-connected graph. Let $\mathcal{G}_2:=\{G: \diam(G)=2\}$ and $\mathcal{G}_3:=\{G:\diam(G)=\diam(\overline{G})=3\}$\aside{$\mathcal{G}_2, \mathcal{G}_3$}. Gupta et al. \cite{GMP03} showed that the Reconstruction Conjecture is true if and only if it is true for every graph in $\mathcal{G}_2\cup\mathcal{G}_3$. Combining the above two results, Monikandan and Ramachandran \cite{MR09} showed that it suffices to consider $2$-connected graphs in $\mathcal{G}_2\cup\mathcal{G}_3$ to prove the Reconstruction Conjecture.

\begin{thm}[\cite{MR09}]
\label{thmA}
The Reconstruction Conjecture is true if and only if every $2$-connected graph in $\mathcal{G}_2\cup\mathcal{G}_3$ is reconstructible.
\end{thm}

As a step towards proving the Reconstruction Conjecture,  Balakumar and Monikandan \cite{BM12} studied the graphs in Theorem~\ref{thmA} which are bipartite and those which are triangle-free. They proved the following:

\begin{thm}[\cite{BM12}]
\label{bipartite-thm}
If $G\in\mathcal{G}_2$ and is bipartite, or if $G\in\mathcal{G}_3$ and is 2-connected and bipartite, then $G$ is reconstructible.
\end{thm}

\begin{thm}[\cite{BM12}]
\label{tfree-thm}
If $G\in\mathcal{G}_2\cup\mathcal{G}_3$ and is triangle-free with $\kappa(G)=2$, then $G$ is reconstructible.
\end{thm}

Note that Theorem~\ref{bipartite-thm} restricts the graphs in Theorem~\ref{thmA} to those containing odd cycles. Furthermore, Theorem~\ref{tfree-thm} gives partial results on the graphs in Theorem~\ref{thmA} which are triangle-free. Balakumar and Monikandan asked whether the class of graphs in Theorem~\ref{tfree-thm} could be extended to those with connectivity at least 3. They remarked that a positive answer to their question would restrict the graphs in Theorem~\ref{thmA} to those containing triangles. Furthermore, narrowing down the classes of graphs critical for proving the Reconstruction Conjecture makes it easier to search for a counterexample, if any. As a partial solution to their question, we prove the following two results.

\begin{thm}\label{diam2}
If $G\in \mathcal{G}_2$ and is triangle-free with $\kappa(G)=3$, then $G$ is reconstructible.
\end{thm}

\begin{thm}
\label{diam3}
If $G\in\mathcal{G}_3$ and is triangle-free with $\kappa(G)\geq3$, then $G$ is reconstructible. 
\end{thm}

Observe that this leaves only the open case of every triangle-free graph $G\in\mathcal{G}_2$ with $\kappa(G)\geq4$.

\medskip

An edge-focused variant of the Reconstruction Conjecture was first proposed by Harary \cite{Har64}. An \emph{edge-card}\aaside{edge-card/deck $\mathcal{E}\mathcal{D}(G)$}{-3mm} of a graph $G$ is $G-e$ for some $e\in E(G)$, and the \emph{edge-deck} of $G$ is the multiset $\mathcal{E}\mathcal{D}(G):=\{G-e: e\in E(G)\}$. A graph $G$ is \EmphE{edge-reconstructible}{3mm} if $G$ is isomorphic to every graph $H$ with $\mathcal{E}\mathcal{D}(H)=\mathcal{E}\mathcal{D}(G)$.
Harary \cite{Har64} proposed the \EmphE{Edge Reconstruction Conjecture}{12mm} which states the following.
\begin{conj}[Edge Reconstruction Conjecture]
Every graph with at least $4$ edges is edge-reconstructible, i.e., it is uniquely determined by its edge-deck.
\end{conj}
Greenwell \cite{Gre71} established a connection between the Reconstruction Conjecture and the Edge Reconstruction Conjecture by showing that the deck of $G$ can be recovered from its edge-deck.
\begin{thm}[\cite{Gre71}]
\label{greenwell}
If $G$ has at least $4$ edges and no isolated vertices, then $\mathcal{D}(G)$ is uniquely determined by  $\mathcal{E}\mathcal{D}(G)$. 
\end{thm}

Theorem \ref{greenwell} implies that a graph $G$ with no isolated vertices and $|E(G)|\ge 4$ is edge-reconstructible if it is reconstructible.

\medskip

We prove edge-reconstruction analogues of Theorems~\ref{diam2} and \ref{diam3} (in fact, we prove a stronger edge-reconstruction analogue of Theorem~\ref{diam2}).

\begin{thm}
\label{edgecon_diam2}
If $G\in\mathcal{G}_2$ and is triangle-free, then $G$ is edge-reconstructible.
\end{thm}

\begin{thm}
\label{edgecon_diam3}
If $G\in\mathcal{G}_3$ and is triangle-free, then $G$ is edge-reconstructible.
\end{thm}

Note that in light of Theorems~\ref{tfree-thm}, \ref{diam3}, and \ref{greenwell}, to prove Theorem~\ref{edgecon_diam3}, it suffices to prove the following.

\begin{thm}\label{conn1diam3}
If $G\in\mathcal{G}_3$ and is triangle-free with $\kappa(G)=1$, then $G$ is reconstructible.
\end{thm}

A more general problem is to decide if a graph parameter is uniquely determined by its deck or edge-deck. Given a graph $G$, a graph parameter $p(G)$ is reconstructible (resp. edge-reconstructible) if the value of $p(G)$ is the same for each graph $H$ with $\mathcal{D}(H)=\mathcal{D}(G)$ (resp. $\mathcal{E}\mathcal{D}(H)=\mathcal{E}\mathcal{D}(G)$). 
A family of graphs $\mathcal{G}$ is \emph{recognizable}\aaside{(edge-) recognizable}{-3mm} if, for each $G\in\mathcal{G}$, every graph $H$ with $\mathcal{D}(H)=\mathcal{D}(G)$ is also in $\mathcal{G}$. Moreover, $\mathcal{G}$ is \emph{weakly reconstructible}\aaside{weakly (edge-) reconstructible}{1mm} if, for each $G\in\mathcal{G}$ and each $H\in\mathcal{G}$ with $\mathcal{D}(H)=\mathcal{D}(G)$, the graph $H$ is isomorphic to $G$. If $\mathcal{G}$ is both recognizable and weakly reconstructible, then it is reconstructible. We define \emph{edge-recognizable} and \emph{weakly edge-reconstructible} analogously.
We will need the following results on reconstructible graph parameters and recognizable graph classes.

\begin{lem}[\cite{LaSc03}]
\label{deletedvertex}
Given a card $G-v$, the degree of $v$ in $G$, as well as the degrees of the neighbors of $v$ in $G$ are reconstructible. Similarly, given an edge-card $G-e$, the degrees in $G$ of the endpoints of $e$ are edge-reconstructible. 
\end{lem}

\begin{lem}[Kelly's Lemma \cite{Kel57}]\label{Kelly}
The number of occurrences of any proper subgraph of $G$ is reconstructible.
\end{lem}

\begin{lem}[\cite{BoHe77}]
\label{connectivity}
The connectivity of $G$ is reconstructible. 
\end{lem}

\begin{lem}[\cite{GMP03}]
\label{diam-recog}
Both $\mathcal{G}_2$ and $\mathcal{G}_3$ are recognizable.
\end{lem}
This paper is organized as follows. In Section \ref{sec:reconstruction}, we prove Theorems~\ref{diam2}, \ref{diam3}, and \ref{conn1diam3} which address our results on reconstruction. In Section \ref{sec:edge-reconstruction}, we prove Theorems \ref{edgecon_diam2} and \ref{edgecon_diam3} which address edge reconstruction.

We end this section with the following note. Ideally, we would like an edge reconstruction result similar to Theorem~\ref{thmA} that would restrict the class of graphs critical for proving the Edge Reconstruction Conjecture to those which lie in $\mathcal{G}_2\cup \mathcal{G}_3$. If such a result is proved, then Theorems~\ref{edgecon_diam2} and \ref{edgecon_diam3} would further restrict the class of critical graphs to those containing triangles. We remark that such a result is unlikely to be proved using techniques similar to those used in the proof of Theorem~\ref{thmA}. In particular, the restricted-diameter result of Gupta et al. \cite{GMP03} relies on the fact that $\mathcal{D}(\overline{G})$ is reconstructible from $\mathcal{D}(G)$ for every graph $G$. This is not true for edge reconstruction. Nevertheless, the advantages of having such a result merit further investigation.

\section{Reconstruction: Proofs of Theorems~\ref{diam2}, \ref{diam3}, and \ref{conn1diam3}}\label{sec:reconstruction}

For every positive integer $k$, denote by $[k]$ the set $\{1,2,\dots,k\}$. Let $G$ be a graph with $\kappa(G)=k\geq3$, and let $S=\{x_1,x_2,\dots,x_k\}$ be a cut set of $G$. For $p\geq2$, denote by $C_1,C_2,\dots,C_p$ the components of $G-S$. A component is \Emph{trivial} if it is a single vertex. We group the vertices in each component into \Emph{classes} based on their neighborhood in $S$. Fix $i\in[p]$ and a vertex $v\in C_i$, and let $\{x_{i_1},x_{i_2},\dots,x_{i_m}\}$ be the neighbors of $v$ in $S$, where $\{i_1,i_2,\dots,i_m\}\subseteq[k]$. We say $v$ is \emph{in the class $C_i(\{i_1,i_2,\dots,i_m\})$}, where $C_i(\{i_1,i_2,\dots,i_m\})$ denotes the subset of vertices in $C_i$ whose neighbors in $S$ are precisely $\{x_{i_1},x_{i_2},\dots,x_{i_m}\}$. For convenience, we write \Emph{$C_i(j)$} for $C_i(\{j\})$. We say the class $C_i(\{i_1,i_2,\dots,i_m\})$ \emph{contains the index $j$} if $j=i_q$ for some $q\in[m]$. If a vertex $v\in C_i$ is not adjacent to any vertex in $S$, then $v$ is in the class \Emph{$C_i(\emptyset)$}. Similarly, if a vertex $v\in C_i$ is adjacent to every vertex in $S$, then $v$ is in the class \Emph{$C_i(S)$}. Note that vertices which are trivial components are adjacent to every vertex in $S$ (by minimality of $S$), i.e., $C_j=C_j(S)$ for every trivial component $C_j$ of $G-S$. Moreover, in a triangle-free graph, every class of a component forms an independent set, except possibly for $C_i(\emptyset)$.  

In their paper \cite{BM12}, Balakumar and Monikandan implicitly introduced the notion of \emph{the classes of a component} when dealing with triangle-free graphs. This helped give more structure to the graph. Their proofs (for the most part) relied on reconstructing the graph from cards which delete a vertex of degree 1 or a cut vertex. Such vertices are ``special" in the sense that their set of neighbors is restricted which narrows down the number of cases to consider. We will use a similar approach. However, we note that the higher the connectivity of the graph is, the less ``special" those vertices become. To avoid this problem, we will try to deduce as much about the structure of the graph as possible (using only what we know about its diameter and connectivity) before reconstructing it.

It follows from Lemmas~\ref{Kelly}, \ref{connectivity}, and \ref{diam-recog} that the class of triangle-free graphs with connectivity $k\in\mathbb{N}$ which lie in $\mathcal{G}_2\cup \mathcal{G}_3$ are recognizable. Therefore, to prove that the graphs in Theorems~\ref{diam2}, \ref{diam3}, and \ref{conn1diam3} are reconstructible, it suffices to show that they are weakly reconstructible.

\begin{proof}[\textbf{Proof of Theorem~\ref{diam2}}]
Let $G$ be a triangle-free graph in $\mathcal{G}_2$ with $\kappa(G)=3$. Pick a card $H$ with connectivity 2 and let $H:=G-x_1$ (such a card exists because $G$ has a cut set of size 3 and some card deletes one of its vertices). Let $\{x_2,x_3\}$ be a cut set of size 2 in $H$. Since $\kappa(G)=3$, the set $S:=\{x_1,x_2,x_3\}$ is a cut set of $G$. 

Let $C_1,\dots,C_p$ be the components of $G-S$. We show that $G-S$ has at most one nontrivial component. By contradiction, assume $C_1$ and $C_2$ are nontrivial. Consider an edge $u_1u_2$ in $C_1$ and an edge $v_1v_2$ in $C_2$. Since $G$ is triangle-free, at least one of $u_1$ and $u_2$ is adjacent to at most one vertex in $S$; say $u_1$. If $u_1$ is adjacent to no vertex in $S$, then $d_G(u_1,v_1)\geq{3}$ contradicting $\diam(G)=2$. So, by symmetry, assume $u_1$ is only adjacent to $x_1$ in $S$. Note that $v_1$ and $v_2$ are not both adjacent to $x_1$; otherwise, we get a triangle. So, assume $v_1$ is nonadjacent to $x_1$. Now, $d_G(u_1,v_1)\geq3$ contradicting $\diam(G)=2$. Thus, $G-S$ has at most one nontrivial component, say $C_1$.

To reconstruct $G$ from $H$, we need to identify the vertices in $H$ adjacent to $x_1$ in $G$. Observe that each vertex in a trivial component of $G-S$ must be adjacent to each vertex of $S$ in $G$ (in particular, it is adjacent to $x_1$); otherwise, $G$ contains a smaller cut set. Moreover, if $C_1$ is trivial, then $G$ is a complete bipartite graph and is reconstructible by Theorem~\ref{bipartite-thm}. Thus, we assume $C_1$ is nontrivial.

\begin{figure}[H]
\centering
\begin{tikzpicture}[scale=0.8, every node/.style={scale=0.8}]
\draw[thick, rounded corners, fill=gray!30!white] (8.5,1) rectangle (15.5,-9);
\draw[thick] (16.2,-4) node[scale=1.5] {$C_1$};
\tkzDefPoint(0:1){C_2}
\tkzDefShiftPoint[C_2](-90:2){C_3}
\tkzDefShiftPoint[C_3](-90:2){dot1}
\tkzDefShiftPoint[dot1](-90:1){dot2}
\tkzDefShiftPoint[dot2](-90:1){dot3}
\tkzDefShiftPoint[dot3](-90:2){C_p}
\draw[thick, rounded corners] (4,-3) rectangle (5.5,-6);
\tkzDefShiftPoint[dot1](0:3.75){x_2}
\tkzDefShiftPoint[x_2](-90:1){x_3}
\draw[thick] (14.2,-1) edge[bend left] (14.2,-4.4) (14.2,-1) edge[bend left] (14.2,-6.8);
\draw[thick] (9.2,-7.7) edge[out=210, in=-30] (x_3) (9.2,-3.5) -- (x_2);
\draw[thick] (C_2) -- (x_2) -- (C_3) (C_p) -- (x_2) (C_2) -- (x_3) -- (C_3) (C_p) -- (x_3);
\draw[thick] (x_3) -- (9,-1) -- (x_2) (10.5,-0.4) edge[out=30, in=150] (13.3,-0.4);
\tkzDrawPoints[scale=2, fill=white, thick](C_2,C_3,C_p)
\tkzDrawPoints[fill=black, thick](dot1,dot2,dot3)
\tkzLabelPoints[scale=1.5, left](C_2,C_3,C_p)
\tkzDrawPoints[scale=2, fill=white, thick](x_2,x_3)
\tkzLabelPoints[scale=1.2,above=0.1cm](x_2)
\tkzLabelPoints[scale=1.2,below=0.1cm](x_3)
\draw[thick, rounded corners, fill =white] (9,-0.4) rectangle (12,-1.6) (12.4,-0.4) rectangle (14.2,-1.6);
\draw[thick] (10.5,-1) node[scale=1.5] {$C_1(\{2,3\})$} (13.3,-1) node[scale=1.5] {$C_1(1)$};
\draw[thick, rounded corners, fill =white] (9,-3.8) rectangle (12.4,-5) (12.4,-3.8) rectangle (14.2,-5);
\tkzDefShiftPoint[x_2](-0.1:5.4){L_{x_2}}
\tkzDefShiftPoint[L_{x_2}](0:2.6){A_2}
\tkzLabelPoints[scale=1.5](L_{x_2},A_2)
\draw[thick, decorate, decoration={calligraphic brace, amplitude=3mm}] (9,-3.7) -- (14.2,-3.7) (11.7,-2.9) node[scale=1.2] {$C_1(2)\cup C_1(\{1,2\})$};
\draw[thick, rounded corners, fill =white] (9,-6.2) rectangle (12.4,-7.4) (12.4,-6.2) rectangle (14.2,-7.4);
\tkzDefShiftPoint[L_{x_2}](-90:2.4){L_{x_3}}
\tkzDefShiftPoint[L_{x_3}](0:2.6){A_3}
\tkzLabelPoints[scale=1.5](L_{x_3},A_3)
\draw[thick, decorate, decoration={calligraphic brace, mirror, amplitude=3mm}] (9,-7.5) -- (14.2,-7.5) (11.7,-8.3) node[scale=1.2] {$C_1(3)\cup C_1(\{1,3\})$};
\draw[thick] (12,2) node[scale=1.5] {$C_1(S)\cup C_1(\emptyset)$};
\path[thick, ->] (10,2) edge[out=180, in=135] (12,0.5);
\end{tikzpicture}

\vspace{1cm}
\centering
\begin{tikzpicture}[scale=0.8, every node/.style={scale=0.8}]
\draw[thick, rounded corners, fill=gray!30!white] (8.5,1) rectangle (15.5,-9);
\draw[thick] (16.2,-4) node[scale=1.5] {$C_1$};
\tkzDefPoint(0:1){C_2}
\tkzDefShiftPoint[C_2](-90:2){C_3}
\tkzDefShiftPoint[C_3](-90:2){dot1}
\tkzDefShiftPoint[dot1](-90:1){dot2}
\tkzDefShiftPoint[dot2](-90:1){dot3}
\tkzDefShiftPoint[dot3](-90:2){C_p}
\draw[thick, rounded corners] (4,-2) rectangle (5.5,-6);
\tkzDefShiftPoint[dot1](0:3.75){x_2}
\tkzDefShiftPoint[x_2](90:1){x_1}
\tkzDefShiftPoint[x_2](-90:1){x_3}
\draw[thick] (14.2,-1) edge[bend left] (14.2,-4.4) (14.2,-1) edge[bend left] (14.2,-6.8);
\draw[thick] (11.5,-6.2) -- (9.8,-5) -- (9.8,-6.2) -- (11.5,-5);
\draw[thick, dashed] (11.5,-5) -- (11.5,-6.2) (9.8,-5) edge[out=-30, in=-150] (11.5,-5) (9.8,-6.2) edge[out=30, in=150] (11.5,-6.2);
\draw[thick] (9.2,-7.7) edge[out=210, in=-30] (x_3) (9.2,-3.5) edge[in=20, out=150] (x_2) (x_1) edge[in=200, out=10] (12.8,-1.6);
\draw[thick] (C_2) -- (x_2) -- (C_3) (C_p) -- (x_2) (C_2) -- (x_3) -- (C_3) (C_p) -- (x_3);
\draw[thick] (x_3) -- (9,-1) -- (x_2) (C_2) -- (x_1) -- (C_3) (x_1) -- (C_p) (10.5,-0.4) edge[out=30, in=150] (13.3,-0.4);
\tkzDrawPoints[scale=2, fill=white, thick](C_2,C_3,C_p)
\tkzDrawPoints[fill=black, thick](dot1,dot2,dot3)
\tkzLabelPoints[scale=1.5, left](C_2,C_3,C_p)
\tkzDrawPoints[scale=2, fill=white, thick](x_1,x_2,x_3)
\tkzLabelPoints[scale=1.2,above=0.1cm](x_2,x_1)
\tkzLabelPoints[scale=1.2,below=0.1cm](x_3)
\draw[thick] (4.75,-1.5) node[scale=1.5] {$S$}; 
\draw[thick, rounded corners, fill =white] (9,-0.4) rectangle (12,-1.6) (12.4,-0.4) rectangle (14.2,-1.6);
\draw[thick] (10.5,-1) node[scale=1.5] {$C_1(\{2,3\})$} (13.3,-1) node[scale=1.5] {$C_1(1)$};
\draw[thick, rounded corners, fill =white] (9,-3.8) rectangle (10.7,-5) (10.7,-3.8) rectangle (12.4,-5) (12.4,-3.8) rectangle (14.2,-5);
\tkzDefShiftPoint[x_2](-0.05:4.5){B_2}
\tkzDefShiftPoint[B_2](0:1.6){B_{12}}
\tkzDefShiftPoint[B_{12}](0:1.9){A_2}
\tkzLabelPoints[scale=1.5](B_2,B_{12},A_2)
\draw[thick, decorate, decoration={calligraphic brace, amplitude=3mm}] (9,-3.7) -- (14.2,-3.7) (11.7,-2.9) node[scale=1.2] {$C_1(2)\cup C_1(\{1,2\})$};
\draw[thick, rounded corners, fill =white] (9,-6.2) rectangle (10.7,-7.4) (10.7,-6.2) rectangle (12.4,-7.4) (12.4,-6.2) rectangle (14.2,-7.4);
\tkzDefShiftPoint[B_2](-90:2.4){B_3}
\tkzDefShiftPoint[B_3](0:1.6){B_{13}}
\tkzDefShiftPoint[B_{13}](0:1.9){A_3}
\tkzLabelPoints[scale=1.5](B_3,B_{13},A_3)
\draw[thick, decorate, decoration={calligraphic brace, mirror, amplitude=3mm}] (9,-7.5) -- (14.2,-7.5) (11.7,-8.3) node[scale=1.2] {$C_1(3)\cup C_1(\{1,3\})$};
\draw[thick] (12,2) node[scale=1.5] {$C_1(S)\cup C_1(\emptyset)$};
\path[thick, ->] (10,2) edge[out=180, in=135] (12,0.5);
\end{tikzpicture}
\captionsetup{skip=30pt} 
\caption{The figure shows the structures of $G$ and $H$ in Theorem~\ref{diam2}. Top: the view in the card $H$. Bottom: the view in $G$. Solid lines between two sets indicate that every vertex in the first set is adjacent to every vertex in the second set. Analogously, dashed lines indicate non-adjacency. The component $C_1$ is divided into classes as defined in Section~\ref{sec:reconstruction}. The set $C_1(S)\cup C_1(\emptyset)$ is grayed out to indicate that it is empty.}
\label{fig-diam2}
\end{figure}
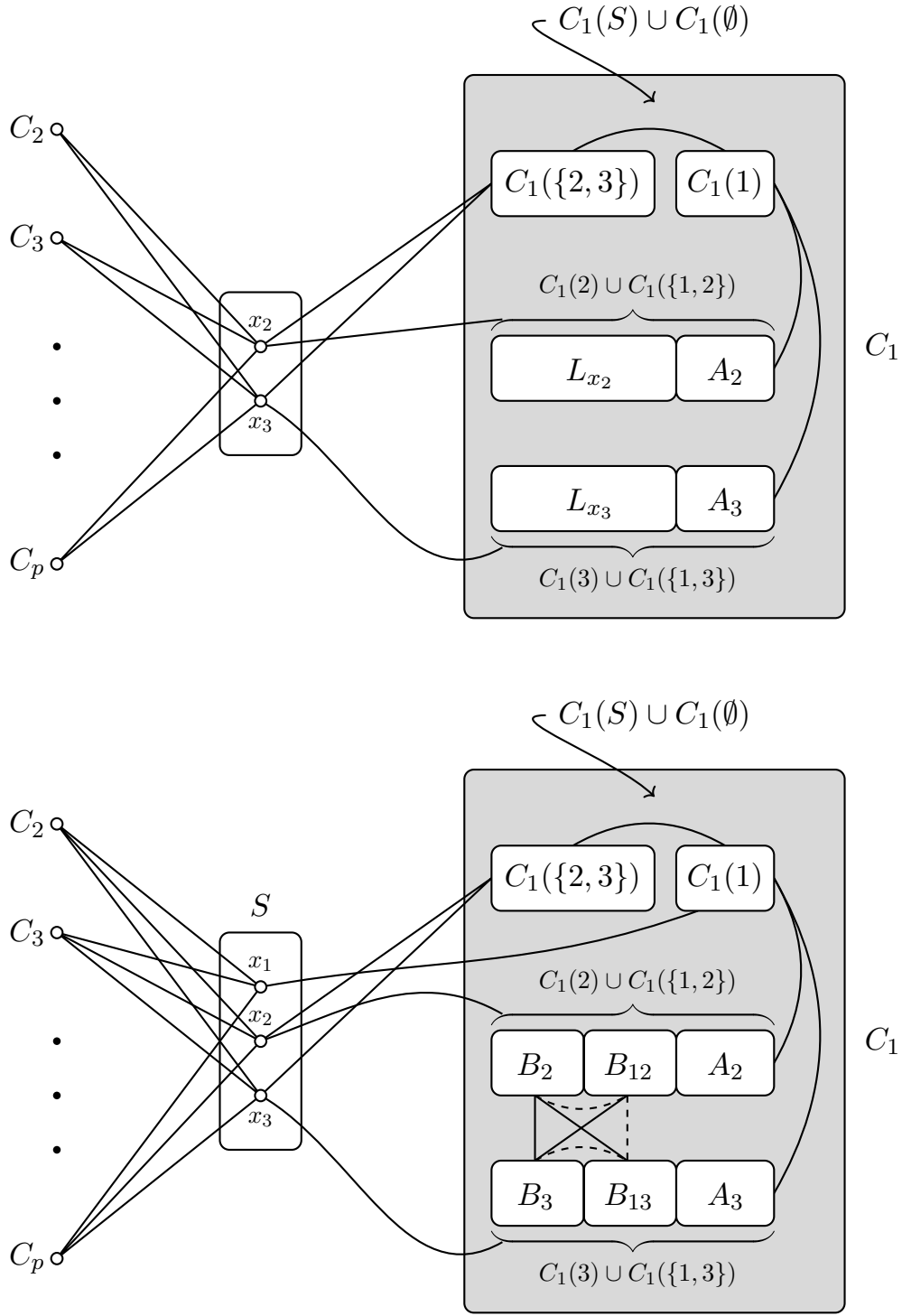

We identify the classes of $C_1$ with respect to the cut set $S$. Note that $C_1(\emptyset) = \emptyset$; otherwise, pick $v\in C_1(\emptyset)$ and observe that $d_G(v,u)\geq3$ for every $u\in C_i$ ($i\neq1)$, a contradiction. Therefore, $C_1(1)$ consists of all vertices in $C_1$ adjacent to neither $x_2$ nor $x_3$ in $H$. Furthermore, $C_1(S)=\emptyset$; otherwise, pick $v\in C_1(S)$ and let $u\in N_{C_1}(v)$. Since $G$ is triangle-free, $u\in C_1(\emptyset)$, a contradiction. Thus, $C_1(\{2,3\})$ consists of all vertices in $C_1$ adjacent to both $x_2$ and $x_3$ in $H$.

This leaves identifying the vertices of $C_1$ in classes $C_1(2), C_1(3), C_1(\{1,2\})$, and $C_1(\{1,3\})$. Let $A_2:=N_{C_1(2)\cup C_1(\{1,2\})}(C_1(1))$ and $A_3:=N_{C_1(3)\cup C_1(\{1,3\})}(C_1(1))$\aside{$A_2,A_3$}. Note that every $v\in A_2\cup A_3$ is nonadjacent to $x_1$, since $G$ is triangle-free. So, it suffices to identify the vertices of each of the following ``$B$" sets. Let $B_2:=C_1(2)-A_2, B_3:=C_1(3)-A_3, B_{12}:=C_1(\{1,2\})-A_2$, and\aside{$B_2,B_3$} $B_{13}:=C_1(\{1,3\})-A_3$; see Figure~\ref{fig-diam2}\aside{$B_{12},B_{13}$}. (Observe that $B_{12}=C_1(\{1,2\})$ and $B_{13}=C_1(\{1,3\})$ since $G$ is triangle-free, so no vertex in $C_1(\{1,2\})$ or $C_1(\{1,3\})$ is adjacent to a vertex in $C_1(1)$.) 

Recall that, since $G$ is triangle-free, (1) each ``$B$" set is independent, and (2) no vertex in $B_{12}$ (resp. $B_{13}$) is adjacent to a vertex in $B_2$ (resp. $B_3$) or a vertex in $B_{13}$ (resp. $B_{12}$). Furthermore, for each $v\in B_2$ and $u\in B_{13}$, there exists no $w\in N_{C_1}(u)\cap N_{C_1}(v)$; otherwise, since $G$ is triangle-free, $w\in C_1(\emptyset)$, a contradiction. Hence, $N_G(u)\cap N_G(v)=\emptyset$ as $N_S(u)\cap N_S(v)=\emptyset$ by definition. Thus, (3) every $v\in B_2$ is adjacent to every $u\in B_{13}$; otherwise, $d_G(u,v)\geq3$, contradicting $\diam(G)=2$. Similarly, (4) every vertex in $B_3$ is adjacent to every vertex in $B_{12}$. Finally, for each $u\in B_2$ and $v\in B_3$, there exists no $w\in N_{C_1}(u)\cap N_{C_1}(v)$; otherwise, $w\in C_1(1)$ and $\{u,v\}\subseteq N_{C_1}(C_1(1))$, contradicting the definitions of $B_2$ and $B_3$. Hence, $N_G(u)\cap N_G(v)=\emptyset$ as $N_S(u)\cap N_S(v)=\emptyset$ by definition. Thus, (5) every $u\in B_2$ is adjacent to every $v\in B_3$; otherwise, $d_G(u,v)\geq3$, contradicting $\diam(G)=2$. 

Now let $L_{x_2}:=B_2\cup B_{12}$ and $L_{x_3}:=B_3\cup B_{13}$\aside{$L_{x_2},L_{x_3}$}. Note that $H[L_{x_2}\cup L_{x_3}]$ is bipartite with parts $L_{x_2}$ and $L_{x_3}$, by (1-2). And, in the card, we are unable to distinguish between $B_2$ and $B_{12}$ in $L_{x_2}$ or $B_3$ and $B_{13}$ in $L_{x_3}$. So, we consider the following cases.  

\ul{\textit{Case 1: $B_{12}$ and $B_{13}$ are both nonempty in $G$.}} Now, in $H$, there must exist vertices $L_2\subseteq L_{x_2}$ that are not adjacent to every vertex in $L_{x_3}$, and vertices $L_3\subseteq L_{x_3}$ that are not adjacent to every vertex in $L_{x_2}$, by (1-5). This means $L_2$ and $L_3$ are precisely $B_{12}$ and $B_{13}$, respectively. Further, $B_2=L_{x_2}-B_{12}$ and $B_3=L_{x_3}-B_{13}$. 

\ul{\textit{Case 2: $B_{12}$ and $B_{13}$ are both empty in $G$.}} Now it must be that $d_{C_1}(x_1)=|C_1(1)|$. Note that we can calculate the value of $d_{C_1}(x_1)$ as follows: $d_{C_1}(x_1)=d_G(x_1)-|\cup_{i\neq1}C_i|$, where $d_G(x_1)$ is given by Lemma~\ref{deletedvertex}. So, in $H$, after identifying the vertices of $C_1(1)$, if $|C_1(1)|=d_{C_1}(x_1)$, then the vertices in $C_1(1)$ are the only neighbors of $x_1$ in $C_1$. This means $B_{12}=B_{13}=\emptyset$, and that $B_2=L_{x_2}$ and $B_3=L_{x_3}$. 

\ul{\textit{Case 3: Neither Case 1 nor Case 2 is true.}} Let $B_{23}$ be defined analogously to $B_{12}$ and $B_{13}$. Note that by Pigeonhole Principle, at least two of the sets $B_{12}, B_{23}$, and $B_{13}$ are nonempty, or at least two of them are empty. Let $i$ be the index shared by those two sets. If $i=1$, then we are done by Case 1 or Case 2. Otherwise, there exists a card $H'$ that deletes $x_i$ and whose corresponding $B_{ij}$ and $B_{ik}$ sets, where $\{j,k\}=\{1,2,3\}-\{i\}$, are either both empty or both nonempty. Now we can repeat the above arguments for $H'$ (where the roles of $H$ and $H'$, and those of $x_1$ and $x_i$ are interchanged).  
\end{proof}

\begin{proof}[\textbf{Proof of Theorem~\ref{diam3}}]
Recall that we only need to show that this class of graphs is weakly reconstructible. Let $G$ be a triangle-free graph in $\mathcal{G}_3$ with $\kappa(G)\geq3$, let $S$ be a cut set of $G$ with $|S|=\kappa(G)$, and let $C_1,\dots,C_p$ be the components of $G-S$. Since $\kappa(G)\geq3$, the set $S$ has at least 3 vertices, say $S=\{x_1,x_2,\dots,x_k\}$ where $k\geq3$. Moreover, since $\diam(\overline{G})=3$, there exist vertices $u,v\in V(G)$ such that $d_{\overline{G}}(u,v)=3$. This means $d_G(u,v)=1$; that is, $uv\in E(G)$. Observe that $d_{\overline{G}}(x,y)\leq2$ for every $x,y\in G-S$. Indeed, if $x$ and $y$ are in the same component $C_i$ of $G-S$ for some $i\in[p]$, then pick $z\in C_j$ for some $j\neq i$. Now, $x,y\in N_{\overline{G}}(z)$ which implies $d_{\overline{G}}(x,y)\leq2$. If, otherwise, $x$ and $y$ are in components $C_i$ and $C_j$ of $G-S$ for some distinct $i,j\in[p]$, respectively, then $xy\in E(\overline{G})$ and $d_{\overline{G}}(x,y)=1$. Therefore, either $v=x_i$ and $u\in G-S$, or $v=x_i$ and $u=x_j$ for some $i,j\in[k]$.

\ul{\textit{Case 1: $S$ is independent.}} This implies $d_{\overline{G}}(x_i,x_j)=1$ for every $i,j\in[k]$. Thus, $u\in G-S$ and $v=x_i$ for some $i\in[k]$. By symmetry, assume $i=1$. Observe that $G-S$ must contain a nontrivial component. Otherwise, $G$ is a complete bipartite graph and is reconstructible by Theorem~\ref{bipartite-thm}. By symmetry, let $C_1$ be a nontrivial component of $G-S$. We show that $u\in C_1(S)$. First, note that $u$ must be adjacent to every vertex in $S$; otherwise, there exists $x_j$ ($j\neq 1$) in $S$ such that $ux_j\notin E(G)$, which implies $ux_j\in E(\overline{G})$. Since $x_1x_j\in E(\overline{G})$, this means $d_{\overline{G}}(x_1,u)=2$, a contradiction. So, $u\in C_t(S)$ for some $t\in[p]$. Since $G$ is triangle-free and $C_1$ is nontrivial, there exists $x\in C_1$ which is not adjacent to $x_1$. So, if $t\neq1$, then $x_1,u\in N_{\overline{G}}(x)$ and $d_{\overline{G}}(x_1,u)=2$, a contradiction. Thus, $u\in C_1(S)$. By a similar argument, $C_q$ is trivial for each $q\neq 1$. Indeed, if $C_q$ is nontrivial for some $q\neq1$, then there exists $x\in C_q$ which is not adjacent to $x_1$ implying that $d_{\overline{G}}(x_1,u)=2$ (through $x$), a contradiction. So, $C_1$ is the only nontrivial component of $G-S$. 

\begin{figure}
\centering
\begin{tikzpicture}[scale=0.8, every node/.style={scale=0.8}]
\draw[thick, rounded corners] (8,1) rectangle (11.2,-0.5) (8,-0.5) rectangle (11.2,-2) (8,-2) rectangle (11.2,-3.5) (8,-3.5) rectangle (11.2,-5) (8,-5) rectangle (11.2,-6.5) (8,-7.8) rectangle (11.2,-9.3);
\draw[thick, decorate, decoration={calligraphic brace, amplitude=3mm}] (11.4,1) -- (11.4,-10.2);
\draw[thick] (12.2,-4.6) node[scale=1.5] {$C_1$};
\draw[thick] (9.6,-6.8) node[fill=black, shape=circle, scale=0.4] {} (9.6,-7.1) node[fill=black, shape=circle, scale=0.4] {} (9.6,-7.4) node[fill=black, shape=circle, scale=0.4] {};
\draw[thick] (9.6,-9.6) node[fill=black, shape=circle, scale=0.4] {} (9.6,-9.9) node[fill=black, shape=circle, scale=0.4] {} (9.6,-10.2) node[fill=black, shape=circle, scale=0.4] {};
\tkzDefPoint(0:1){C_2}
\tkzDefShiftPoint[C_2](-90:2){C_3}
\tkzDefShiftPoint[C_3](-90:2){dot1}
\tkzDefShiftPoint[dot1](-90:1){dot2}
\tkzDefShiftPoint[dot2](-90:1){dot3}
\tkzDefShiftPoint[dot3](-90:2){C_p}
\draw[thick, rounded corners] (4,-1) rectangle (5.5,-7);
\tkzDefShiftPoint[C_3](0:3.75){x_1}
\tkzDefShiftPoint[x_1](-90:1){x_2}
\tkzDefShiftPoint[x_2](-90:1){dot4}
\tkzDefShiftPoint[dot4](-90:0.5){dot5}
\tkzDefShiftPoint[dot5](-90:0.5){dot6}
\tkzDefShiftPoint[dot6](-90:1){x_k}
\draw[thick] (C_2) -- (x_1) -- (C_3) (x_2) -- (C_p) -- (x_1) (C_2) -- (x_2) -- (C_3) (C_2) -- (x_k) -- (C_3) (x_k) -- (C_p);
\draw[thick] (8,-1.25) -- (x_1) -- (8,-2.75) (8,-4.25) -- (x_1) -- (8,-5.75) (8,-8.55) -- (x_1) (8,-4.25) -- (x_2) -- (8,-1.25) (x_2) -- (8,-8.55) (8,-1.25) -- (x_k);
\tkzDrawPoints[scale=2, fill=white, thick](C_2,C_3,C_p)
\tkzDrawPoints[fill=black, thick](dot1,dot2,dot3,dot4,dot5,dot6)
\tkzLabelPoints[scale=1.5, left](C_2,C_3,C_p)
\tkzDrawPoints[scale=2, fill=white, thick](x_1,x_2,x_k)
\tkzLabelPoints[scale=1.2,above=0.1cm](x_1)
\tkzLabelPoints[scale=1.2,below=0.2cm](x_2)
\tkzLabelPoints[scale=1.2,below=0.1cm](x_k)
\draw[thick] (9.6,0.25) node[scale=1.5] {$C_1(\emptyset)$} (9.6,-1.25) node[scale=1.5] {$C_1(S)$} (9.6,-2.75) node[scale=1.5] {$C_1(1)$} (9.6,-4.25) node[scale=1.5] {$C_1(\{1,2\})$} (9.6,-5.75) node[scale=1.5] {$C_1(\{1,3\})$} (9.6,-8.55) node[scale=1.5] {$C_1(\{1,2,3\})$};
\end{tikzpicture}
\captionsetup{skip=30pt} 
\caption{The figure shows the structure of $G$ in Case 1 of Theorem~\ref{diam3}.}
\label{fig-diam3}
\end{figure}
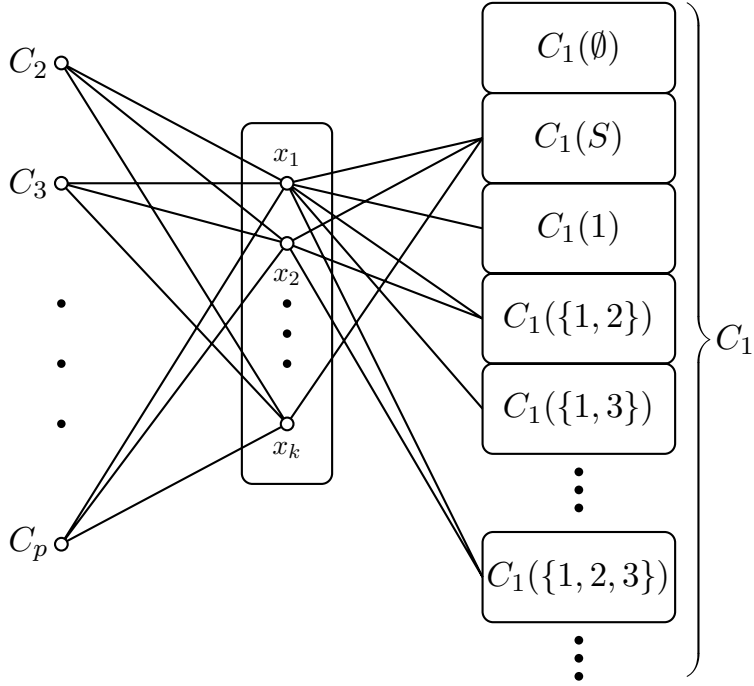

Note that $C_1(\emptyset)\neq\emptyset$. To see this, recall that $C_1(S)\neq\emptyset$ and is an independent set, $C_1$ is a nontrivial component, and $G$ is triangle-free. Therefore, the neighbors in $C_1$ of vertices in $C_1(S)$ can only be in $C_1(\emptyset)$, i.e., $C_1(\emptyset)\neq\emptyset$. Further, every nonempty class in $C_1$ (except for $C_1(\emptyset)$) contains the index 1 (see Figure~\ref{fig-diam3}); otherwise, as before, there exists a vertex $x$ not adjacent to $x_1$ and $x_1,u\in N_{\overline{G}}(x)$ contradicting $d_{\overline{G}}(x_1,u)=3$. Finally, $u$ must be adjacent to every vertex in $C_1(\emptyset)$; otherwise, there exists $x\in C_1(\emptyset)$ not adjacent to $u$, so $x_1,u\in N_{\overline{G}}(x)$ contradicting $d_{\overline{G}}(x_1,u)=3$. Since $G$ is triangle-free, this implies $C_1(\emptyset)$ is independent. Observe that $G$ is now bipartite as follows: Let every class of $C_1$ except $C_1(\emptyset)$ be in one part along with the trivial components of $G-S$, and let $C_1(\emptyset)$ and $S$ be in the other part. It is easy to check that this forms a bipartition of $G$. Thus, $G$ is reconstructible by Theorem~\ref{bipartite-thm}.

\ul{\textit{Case 2: $S$ contains an edge.}} Assume first that $v=x_i$ for some $i\in[k]$ and $u\in G-S$. Let $C_j$ be the component containing $u$ with $j\in[p]$. Since $S$ contains an edge and $G$ is triangle-free, every component of $G-S$ is nontrivial. Pick a component $C_t$ with $t\neq j$. Since $C_t$ is nontrivial, there exists $x\in C_t$ that is not adjacent to $x_i$. Now $x_i,u\in N_{\overline{G}}(x)$ contradicting $d_{\overline{G}}(x_i,u)=3$. Thus, we may assume $v=x_i$ and $u=x_j$ for some $i,j\in[k]$. By symmetry, assume $i=1$ and $j=2$. Since $x_1x_2\in E(G)$ and $G$ is triangle-free, no vertex is adjacent to both $x_1$ and $x_2$. Moreover, every vertex is adjacent to at least one of $x_1$ and $x_2$. Indeed, if some vertex $x$ is nonadjacent to both $x_1$ and $x_2$, then $x_1,x_2\in N_{\overline{G}}(x)$ contradicting $d_{\overline{G}}(x_1,x_2)=3$. Observe that $G$ is again bipartite as follows: Let $N_G(x_1)$ be one part and $N_G(x_2)$ be the other part. Thus, $G$ is reconstructible by Theorem~\ref{bipartite-thm}.
\end{proof}

Note that in light of Theorem \ref{thmA}, Theorem \ref{conn1diam3}  is not needed in order to prove the Reconstruction Conjecture. However, we will refer to this theorem when we consider edge reconstruction in the next section.

\begin{proof}[\textbf{Proof of Theorem~\ref{conn1diam3}}]
Recall that we only need to show that this class is weakly reconstructible. Let $G$ be a a triangle-free graph in $\mathcal{G}_3$ with $\kappa(G)=1$, let $x$ be a cut vertex of $G$, and let $C_1,\dots,C_p$ be the components of $G-x$.

\begin{claim}
\label{one-comp}
$G-x$ has exactly one nontrivial component, say $C_1$.
\end{claim}

\begin{proof}[Proof of Claim~\ref{one-comp}]
If $G-x$ has no nontrivial components, then $G$ is a star contradicting $\diam(G)=~3$. Suppose instead that $C_1$ and $C_2$ are two nontrivial components of $G-x$. Since $G$ is triangle-free, $C_i$ contains at least one vertex, $v_i$, that is not adjacent to $x$ for each $i\in\{1,2\}$. Now $d_G(v_1,v_2)\geq4$, a contradiction. Hence, $G-x$ has exactly one nontrivial component, $C_1$, as desired.
\end{proof}

Let $C_1^x(\emptyset):=C_1(\emptyset)$ with respect to cut vertex $x$.\aside{$C_1^x$}

\begin{claim}
\label{bipartite-C1}
$C_1$ is bipartite with parts $C_1^x(\emptyset)$ and $C_1(x)$. Furthermore, $G$ is bipartite with parts $C_1^x(\emptyset)\cup \{x\}$ and $C_1(x)\cup T_x$, where $T_x$ denotes the set of all vertices in the trivial components of $G-x$.
\end{claim}

\begin{proof}[Proof of Claim~\ref{bipartite-C1}]
Since $G$ is triangle-free, $C_1(x)$ is independent. Also, note that $C_1^x(\emptyset)\neq\emptyset$ because $C_1$ is nontrivial. Since $\diam(\overline{G})=3$, there exist $u,v\in V(G)$ such that $d_{\overline{G}}(u,v)=3$. If $u,v\in C_i$ for some $i\in[p]$, then pick $w\in C_j$ for $j\neq i$. Now $u,v\in N_{\overline{G}}(w)$ and $d_{\overline{G}}(u,v)\leq2$, a contradiction. Similarly, if $u\in C_i$ and $v\in C_j$ for some $i\neq j$, then $uv\in E(\overline{G})$ and $d_{\overline{G}}(u,v)=1$, a contradiction. So, we assume that $u=x$ and $v\in G-x$. Observe that $v\notin C_1^x(\emptyset)$ and $v\notin T_x$; otherwise, $xv\in E(\overline{G})$ and $x,v\in N_{\overline{G}}(w)$ for some $w\in C_1^x(\emptyset)$, respectively. In both cases, $d_{\overline{G}}(x,v)\leq2$, a contradiction. It follows that $v\in C_1(x)$. Furthermore, if there exists $w\in C_1^x(\emptyset)$ such that $vw\notin E(G)$, then $x,v\in N_{\overline{G}}(w)$ and $d_{\overline{G}}(x,v)=2$, a contradiction. Hence, $C_1^x(\emptyset)\subseteq N_G(v)$ which implies $C_1^x(\emptyset)$ is independent since $G$ is triangle-free. Now observe that $C_1$ is bipartite with parts $C_1^x(\emptyset)$ and $C_1(x)$. Further, $G$ is bipartite with parts $C_1^x(\emptyset)\cup \{x\}=N_G(v)$ and $C_1(x)\cup T_x=N_G(x)$, as desired; see Figure~\ref{fig-conn1diam3}. 
\end{proof}

\begin{figure}
\centering
\begin{tikzpicture}[scale=0.8, every node/.style={scale=0.8}]
\draw[thick, rounded corners] (6,-2) rectangle (9,-4) (6,-4) rectangle (9,-6);
\draw[thick, decorate, decoration={calligraphic brace, mirror, amplitude=3mm}] (9.2,-6) -- (9.2,-2);
\draw[thick] (10,-4) node[scale=1.5] {$C_1$};
\tkzDefPoint(0:1){C_2}
\tkzDefShiftPoint[C_2](-90:2){C_3}
\tkzDefShiftPoint[C_3](-90:2){dot1}
\tkzDefShiftPoint[dot1](-90:1){dot2}
\tkzDefShiftPoint[dot2](-90:1){dot3}
\tkzDefShiftPoint[dot3](-90:2){C_p}
\tkzDefShiftPoint[dot1](0:3){x}
\draw[thick] (C_2) -- (x) -- (C_3) (C_p) -- (x) -- (6,-3);
\tkzDrawPoints[scale=2, fill=white, thick](C_2,C_3,C_p)
\tkzDrawPoints[fill=black, thick](dot1,dot2,dot3)
\tkzLabelPoints[scale=1.5, left](C_2,C_3,C_p)
\tkzDrawPoints[scale=2, fill=white, thick](x)
\tkzLabelPoints[scale=1.2](x)
\draw[thick] (7.5,-3) node[scale=1.5] {$C_1(x)$};
\draw[thick] (7.5,-5) node[scale=1.5] {$C_1^x(\emptyset)$};
\end{tikzpicture}
\captionsetup{skip=30pt} 
\caption{The figure shows the structure of $G$ in Theorem~\ref{conn1diam3}.}
\label{fig-conn1diam3}
\end{figure} 

We now split the rest of the proof into three cases. By slight abuse of notation, we refer to the nontrivial component for any cut vertex as $C_1$.

\ul{\textit{Case 1: There exists a cut vertex $x$ of $G$ such that $|C_1^x(\emptyset)|\ne |C_1(x)|$.}} By Claims~\ref{one-comp} and \ref{bipartite-C1}, this means that there exists a card $H$ of $G$ such that $H$ is disconnected and the only nontrivial component $C_1$ of $H$ is bipartite with unequal parts. Let $H=G-y$ for some cut vertex $y\in V(G)$. By the above claims, $N_G(y)$ consists of all trivial components of $H$, as well as, all vertices in one part of $C_1$. By Lemma~\ref{deletedvertex}, we can recover $d_G(y)$. Then, since the parts of $C_1$ are unequal, we can identify $N_{C_1}(y)$. Thus, $G$ is reconstructible.

\ul{\textit{Case 2: $|C_1^{x'}(\emptyset)|=|C_1(x')|$ for every cut vertex $x'$ of $G$ and there exists a cut vertex $x$ of $G$ such that $G-x$ has $k$ trivial components for some integer $k\ge 2$.}} We claim that $x$ is the only cut vertex of $G$. Indeed, assume some $x'\neq x$ is another cut vertex of $G$ where $G-x'$ has $k'\geq1$ trivial components. Note that $x'\in C_1(x)$ since $G-v$ is connected for every $v\in C_1^x(\emptyset)\cup T_x$. This means $N_G(x')\subseteq C_1^x(\emptyset)\cup \{x\}$, i.e., $d_G(x')\le |C_1^x(\emptyset)|+1$. Further, $d_G(x')=|C_1(x')|+k'$ and $|C_1^{x'}(\emptyset)|=|C_1(x')|=d_G(x')-k'$. Since $|V(G)|=|C_1^{x'}(\emptyset)|+|C_1(x')|+1+k'=|C_1^x(\emptyset)|+|C_1(x)|+1+k$, we have $2(d_G(x')-k')+1+k'=2|C_1^x(\emptyset)|+1+k$ which implies $d_G(x')=|C_1^x(\emptyset)|+(k+k')/2\ge |C_1^x(\emptyset)|+(2+1)/2>|C_1^x(\emptyset)|+1$, a contradiction. Thus, $x$ is the only cut vertex of $G$.

We can identify that $G$ has a unique cut vertex by checking that every card in the deck is connected except for one. Then we pick a card $H=G-z$ for some $z\in V(G)$ with $d_G(z)=1$ (again, $d_G(z)$ is reconstructible by Lemma~\ref{deletedvertex}). Note that the unique neighbor of any degree 1 vertex in $G$ is a cut vertex, and deleting a degree 1 vertex in $G$ does not create a new cut vertex in the card. So, the unique cut vertex of $G$ is still unique in $G-z$ and is the neighbor of $z$. Thus, $G$ is reconstructible.

\ul{\textit{Case 3: $|C_1^x(\emptyset)|=|C_1(x)|$ and $G-x$ has exactly one trivial component for every cut vertex $x$ of $G$.}} Note that $|V(G)|$ is easily reconstructible from the cards of $G$ since each card deletes a single vertex. Since $G-x$ contains a single trivial component, it follows that $|C_1^x(\emptyset)|=|C_1(x)|=(|V(G)|-2)/2$ for every cut vertex $x$ of $G$. By Claim~\ref{bipartite-C1}, $G$ is bipartite with parts $C_1^x(\emptyset)\cup \{x\}$ and $C_1(x)\cup T_x$, where $|T_x|=1$, which implies $|C_1^x(\emptyset)\cup \{x\}|=|C_1(x)\cup T_x|=(|V(G)|-2)/2+1=|V(G)|/2$. Further, $|C_1^x(\emptyset)\cup \{x\}|=|C_1(x)\cup T_x|\geq2$ since $C_1^x(\emptyset)\cup \{x\}$ contains $x$ and at least one non-neighbor of $x$. 

Observe that for every cut vertex $x$ of $G$, the trivial component of $G-x$ has degree one in $G$. So, there exists a connected bipartite card $H=G-y$ for some $y\in V(G)$ such that $d_G(y)=1$ (again, $d_G(y)$ is reconstructible by Lemma~\ref{deletedvertex}). Note that $y$ is the  unique trivial component of $G-z$ for some cut vertex $z$ of $G$; in particular, $z$ is the unique neighbor of $y$ in $G$. Since $H$ is connected, it can be uniquely bipartitioned into parts $X$ and $Y$ with $|X|>|Y|$ and $|X|-|Y|=1$, by the above arguments. Thus, $y$ has no neighbors in $Y$ and the unique neighbor $z$ of $y$ is a vertex in $X$ such that $N_H(z)=Y$. If there exist distinct vertices $x_1$ and $x_2$ in $X$ such that $N_H(x_1)=N_H(x_2)=Y$, then we pick $z$ arbitrarily between $x_1$ and $x_2$ since $H+x_1y$ and $H+x_2y$ are isomorphic. Thus, $G$ is reconstructible.
\end{proof}

\section{Edge Reconstruction: Proofs of Theorems~\ref{edgecon_diam2} and \ref{edgecon_diam3}}\label{sec:edge-reconstruction}

In this section, we consider edge reconstruction for triangle-free graphs in $\mathcal{G}_2\cup\mathcal{G}_3$. The proof of Theorem~\ref{edgecon_diam2} uses a systematic approach to try and identify the endpoints of the deleted edge. On the other hand, the proof of Theorem~\ref{edgecon_diam3} follows easily from previous results. 


\begin{proof}[\textbf{Proof of Theorem~\ref{edgecon_diam2}}]
Note that this class of graphs is edge-recognizable by Lemmas~\ref{Kelly} and \ref{diam-recog} and Theorem~\ref{greenwell}. So, we only need to show it is weakly edge-reconstructible. Let $G$ be a triangle-free graph in $\mathcal{G}_2$. For every $uv\in E(G)$, if $d_{G-uv}(x,y) \geq 3$ for some $x,y\in V(G-uv)$, then $x=u$ and $y \in N_G[v]$, or $x=v$ and $y \in N_G[u]$; in particular, $d_{G-uv}(u,v) \geq 3$ since $G$ is triangle-free. To see this, note that each pair $x,y\in V(G-uv)$ with $\{x,y\}\neq \{u,v\}$ and $d_{G-uv}(x,y)\geq3$ must be nonadjacent in $G$ and must use the edge $uv$ in $G$ to satisfy $d_G(x,y)=2$. 

Pick an edge-card $H=G-uv$ for some $uv\in E(G)$ and let \aside{$P_H$}$P_H:=\{\{x,y\}: x,y\in V(H)\text{ and }\\ d_H(x,y)\ge3\}$. Assume first that $P_H=\{\{x,y\}\}$, i.e., $|P_H|=1$. Now, $\{x,y\}=\{u,v\}$ and $G=H+xy$ since $\{u,v\}\in P_H$. Assume instead that $|P_H|=k\ge3$. If $\{x,y\}, \{x,z\}\in P_H$, then $x=u$ or $x=v$; otherwise $\{y,z\}=\{u,v\}$ and $x\in N_G(u)\cap N_G(v)$, i.e., $x$, $y$, and $z$ form a triangle, a contradiction. Moreover, since $\{u,v\}\in P_H$, each pair in $P_H$ contains $u$ or $v$, and $k\ge3$, there exist at least two pairs in $P_H$ with a common vertex $x$. So, $x=u$ or $x=v$; say $x=u$. If there exists another vertex $y$ which also appears in more than one pair in $P_H$, then $(x,y)=(u,v)$ and $G=H+xy$. So, suppose $x$ is the only vertex that appears in more than one pair in $P_H$. If $\{y,z\}\in P_H$ with $y\neq x$ and $z\neq x$, then either $y=v$ or $z=v$ (since each pair in $P_H$ contains $u$ or $v$), say $y=v$. Now, $\{x,y\}=\{u,v\}\in P_H$ which means $y$ appears in more than one pair in $P_H$, a contradiction. So, suppose $x$ appears in every pair in $P_H$, i.e., $P_H=\{\{x,y_1\},\{x,y_2\},\dots,\{x,y_k\}\}$. Now, $y_i\in N_G[v]$ for each $i\in[k]$. More precisely, the set $\{y_1,y_2,\dots,y_k,x\}$ forms an induced star in $H$ with center $v=y_i$ for some $i\in[k]$; so, $G=H+xy_i$. Thus, we may assume $|P_H|=2$ for every edge-card $H$.

By symmetry, assume $P_H=\{\{u,v_1\},\{u,v_2\}\}$ where $v\in\{v_1,v_2\}$. Let $w=\{v_1,v_2\}-v$, i.e., $\{v_1,v_2\}=\{v,w\}$. Note that $v_1v_2\in E(H)$ (and therefore, $v_1v_2\in E(G)$). If $d_H(v_1)\neq d_H(v_2)$, then we can identify $v$ since $d_G(v)$ is edge-reconstructible by Lemma~\ref{deletedvertex}. So, assume that $d_H(v_1)=d_H(v_2)$. Now, there exists an edge-card $H'$ which deletes $v_1v_2$. Interchanging the roles of $u$ and $w$ in the above arguments, $P_{H'}=\{(v,w),(u,w)\}$ since $d_H(u,w)\ge3$ (i.e, $v$ is the only common neighbor of $u$ and $w$ in $G$). So, as in $H$, we may assume $d_{H'}(u)=d_{H'}(v)$. Hence, $d_G(u)=d_G(w)=d_G(v)-1$. This defines a bijection $\sigma:E(G) \rightarrow E(G)$ such that, for every $ab \in E(G)$, there exists $bc \in E(G)$ with $\sigma(ab)=bc$, $\sigma(bc)=ab$, and $d_G(a)=d_G(c)=d_G(b)-1$, where $a,b,c\in V(G)$ and $b$ is the only common neighbor of $a$ and $c$ in $G$. Since every edge in $G$ connects vertices whose degrees differ by one and, therefore, are of different parity, $G$ is bipartite. Thus, $G$ is edge-reconstructible by Theorems~\ref{bipartite-thm} and \ref{greenwell}. 
\end{proof}

\begin{proof}[\textbf{Proof of Theorem~\ref{edgecon_diam3}}]
As before, we only need to show that this class of graphs is weakly edge-reconstructible. Let $G$ be a triangle-free graph in $\mathcal{G}_3$. Since $\diam(G)$ is finite, $G$ is connected. If $\kappa(G)=1$, then we are done by Theorems~\ref{greenwell} and \ref{conn1diam3}. If $\kappa(G)=2$, then we are done by Theorems~\ref{tfree-thm} and \ref{greenwell}. Finally, if $\kappa(G)\geq3$, then we are done by Theorems~\ref{diam3} and \ref{greenwell}. Observe that the result remains true even if $|E(G)|<4$ since the only such graph is $P_4$, and $P_4$ is edge-reconstructible as no other graph can have an edge-card isomorphic to $2K_2$.
\end{proof}


\section{Acknowledgments}
This project began at the Graduate Research Workshop in Combinatorics in 2021. We sincerely thank the organizers. Alexander Clifton was supported by the Institute for Basic Science (IBS-R029-C1) and partially supported by NSF award DMS-1945200. Xiaonan Liu was partially supported by NSF award DMS-1856645 and NSF award DMS-1954134.

\bibliographystyle{plainurl}
\footnotesize{
\bibliography{ref}
}	 

\end{document}